\documentclass[a4paper,12pt, onesided]{amsart}

\usepackage[latin1]{inputenc}
\usepackage{amsfonts,amsmath,amssymb,amsthm}
\usepackage[english]{babel}
\usepackage{color}
\usepackage[margin = 3cm, marginparwidth = 2.4cm]{geometry}
\usepackage{pinlabel}

\newtheorem{thm}{Theorem}[section]
\newtheorem{lemma}[thm]{Lemma}

\newtheorem{cor}[thm]{Corollary}
\newtheorem{defn}[thm]{Definition}

\theoremstyle{remark}
\newtheorem{ex}[thm]{Example}

\numberwithin{figure}{section}
\numberwithin{equation}{section}

\newcommand{\Z}{\mathbb{Z}}
\newcommand{\F}{\mathbb{F}}

\newcommand{\Q}{\mathbb{Q}}
\renewcommand{\P}{\mathbb{P}}
\newcommand{\cP}{\mathcal{P}}
\newcommand{\CP}{\mathbb{CP}\vphantom{\mathbb{P}}^2}

\renewcommand{\L}{\mathcal{L}}

\newcommand{\A}{\mathcal{A}}

\newcommand{\del}{\partial}
\newcommand{\co}{\colon}
\newcommand{\fs}{\mathfrak{s}}
\newcommand{\ft}{\mathfrak{t}}

\title{Topological line arrangements with high multiplicities}
\author{Paolo Aceto}
\address{Laboratoire Paul Painlev\'e, Universit\'e de Lille, France}
\email{paolo.aceto@univ-lille.fr}

\author{Marco Golla}
\address{CNRS and Laboratoire de Math{\'e}matiques Jean Leray, Nantes Universit\'e, France}
\email{marco.golla@univ-nantes.fr}

\date{}

\begin{document}

\begin{abstract}
We investigate constraints on the existence of topological and smooth realisations of combinatorial line arrangements and $(n_k)$-configurations in the complex projective plane $\mathbb{CP}^2$. By replacing complex lines with locally-flatly or smoothly embedded 2-spheres, we explore the extent to which classical geometric results, such as Hirzebruch's inequality, persist in the topological or smooth category.

   We introduce two classes of special line arrangements that we call \emph{odd} and \emph{even}. 
   We provide constraints for any smoothly realised, non-trivial, \emph{odd} arrangement via Furuta's 10/8-Theorem.  By looking at branched double covers and using the $G$-signature theorem, we study topologically realised, non-trivial, \emph{even} arrangement. Finally, we establish a new lower bound for $(n_k)$-configurations, showing that any topologically realised configuration must satisfy $n \geq k^2-5$, which implies the non-existence of topological realisations for finite projective planes. 
\end{abstract}

\maketitle

\section{Introduction}
A classical topic in geometry is the study of line arrangements and configurations (viewed as combinatorial objects) and their geometric realisations via points and lines in the real or complex projective plane. In this paper we investigate topological realisations of arrangements and configurations with the general goal of understanding how current techniques in 4-manifold topology compare with more classical geometric results.

A \emph{line arrangement} is given by a pair of finite sets $(\mathcal{L},\mathcal{P})$ of lines and points and an incidence relation 
$I\subset\mathcal{L}\times\mathcal P$ such that for any pair of lines $\ell,\ell'\in\mathcal{L}$ there exists a unique point $p\in\mathcal{P}$ with $(\ell,p)=(\ell',p)\in I$. When $(\ell,p)\in I$ we say that $\ell$ is \emph{incident} to $p$ or that $p$ \emph{lies} on $\ell$. The \emph{multiplicity} of a point is the number of lines that are incident to it. For any $k\geq 2$ we denote by $t_k(\A)$ the number of points of multiplicity $k$ of the line arrangement $\A$. We will omit $\A$ from the notation when there is no risk of confusion.

A \emph{pencil} of $d$ lines is a line arrangement with $t_d=1$ (and therefore $t_m=0$ when $m\neq d$).
An \emph{almost-pencil} of $d$ lines is a line arrangement with $t_{d-1}=1$, and $t_2 = d-1$.
We will say that a line arrangement is \emph{non-trivial} if it is neither a pencil nor an almost-pencil.

A \emph{geometric realisation} of a line arrangement $\A$ is a collection of lines in $\mathbb{CP}^2$ which realises the incidence relations of $\A$. One way to study geometric realisations is by looking at constraints on the $t_m$'s. A classical result of Hirzebruch \cite{Hirzebruch:1983} states that, for any geometric realisation of a non-trivial line arrangement with $d$ lines one has
\[
t_2+t_3\geq d+\sum_{m\geq 4}(m-4)t_m.
\]
In particular, any such realisation has at least one double or triple point. Hirzebruch's result is obtained by studying suitable branched covers of $\CP$ and using the BMY inequality, and therefore its proof does not have an immediate topological meaning.

By replacing lines with locally-flatly (respectively, smoothly),
pairwise transversely,
embedded 2-spheres in $\mathbb{CP}^2$ we obtain the notion of topological (resp. smooth) realisation of a line arrangement.
In particular, each 2-sphere is necessarily homologous to a complex line.

It is natural to ask if in the more flexible topological setting
one can have more degenerate realisations. For instance: is it still true that one always has at least one double or triple point? In this generality, the problem seems to be out of reach. However, constraints analogous to Hirzebruch's inequality can be obtained in particular cases. We call a line arrangement \emph{even} if it contains an even number of lines and $t_k=0$ whenever $k$ is odd.  An arrangement is \emph{odd} if $t_k=0$ whenever $k$ is even.

\begin{thm}\label{thm:oddbound}
    For any smoothly realised, non-trivial, odd line arrangement  we have
    \[
    t_3+t_5+t_7>0.
    \]
\end{thm}

This result is a consequence of a more general inequality in the style of Hirzebruch's results (Theorem~\ref{t:odd} below). We consider the 4-manifold obtained by blowing up all intersection points of the configuration. We then remove tubular neighbourhoods of the proper transforms of all the lines and glue back in suitable positive-definite linear plumbings of spheres. This results in a closed, spin 4-manifold to which we apply Furuta's 10/8-Theorem. The use of Furuta's result explains why we have to work in the smooth category. Along the way, a careful analysis of the spin and spin$^c$ structures involved in the construction is required. To do this we make use of correction terms from Heegaard Floer homology.

\begin{thm}\label{thm:evenbound}
    For any topologically realised, non-trivial, even line arrangement we have
    \[
    t_2+t_4>0.
    \]
\end{thm}

For this results our strategy is to blow up
all intersection points as before and then to consider a suitable double cover (relying on the fact that the arrangement is even). We use the G-signature theorem to compute the signature of this cover. Combining this computation with an estimate of the Betti numbers of the cover eventually leads to another Hirzebruch-type inequality (Theorem~\ref{t:divisible} below).

Let $n, k$ be two positive integers. An $(n_k)$-\emph{configuration} is a combinatorial line arrangement $(\cP, \L)$ with $n$ lines, such that there exist $\cP' \subset \cP$ comprising $n$ points, such that each point in $\cP'$ is contained in exactly $k$ lines, and each line in $\L$ contains exactly $k$ points in $\cP'$.

Just like we did for line arrangements, we reserve the term ``configuration'' for the combinatorial data. One can then talk about geometric, smooth, or topological realisations. A standard reference for configurations is Grunbaum's monograph~\cite{Grunbaum:2009}. The main question we are interested in is: given $k$, what is the smallest $n$ such that an $(n_k)$-configuration is realised?

It is very easy to see that for any $(n_k)$-configuration $(\cP,\L)$, we have $n \ge k^2-k+1$.
Pick a point in $\cP'$: each of the $k$ lines that pass through it contains at least $k-1$ other points of $\cP'$, so we have at least $k(k-1)+1$ points in $\cP$. Configurations with $n=k^2-k+1$
are known as \emph{finite projective planes}.
One way to construct finite projective planes  is via projectivisation of 3-dimensional vector spaces over finite fields. These will be denoted by $\mathbb{F}_q\mathbb{P}^2$, with $q$ a prime power. In \cite{Ruberman-Starkston:2016} Ruberman and Starkston have considered the problem of topological realisations of configurations. They have shown that finite projective planes of the form $\mathbb{F}_q\mathbb{P}^2$ are not topologically realisable. 

\begin{thm}
   For any topologically realised $(n_k)$-configuration we have
   \[
   n\geq k^2-5.
   \]
   In particular, no finite projective plane is topologically realisable.
\end{thm}

It is worth mentioning that in \cite{Aceto-McCoy-Park:2026} the first author, D. McCoy and JungHwan Park use different techniques to show that if an $(n_k)$-configuration is \emph{smoothly} realisable then $n\geq k^2-1$.
Note also that Hirzebruch's inequality implies that for \emph{complex} line arrangements $n \ge k^2+k-5$ (see \cite[Proposition 5.1]{Aceto-McCoy-Park:2026} for a detailed proof).

Our results together with those from \cite{Ruberman-Starkston:2016} and \cite{Aceto-McCoy-Park:2026} show that different bounds can be obtained according to whether one relies on tools from the topological or smooth category. It would be interesting to test this gap by producing explicit configurations which can be realised topologically but not smoothly. 

\subsection*{Organisation} The paper is organised as follows. In Section~\ref{s:combinatorics}, we establish the combinatorial foundations of line arrangements, fixing notation for multiplicities and reviewing essential counting lemmas. In Section~\ref{s:even},
we define the class of $q$-divisible arrangements, which generalise even arrangements, and derive constraints on their realisability by constructing branched covers and applying the $G$-signature theorem. Section~\ref{s:odd} is devoted to the study of odd arrangements; here, we employ smooth 4-manifold techniques, specifically Furuta's 10/8-Theorem and Heegaard Floer correction terms, to obtain multiplicity bounds. Finally, in Section~\ref{s:configurations}, we apply our results to $(n_k)$-configurations and finite projective planes, establishing new topological lower bounds and proving the non-realisability of finite projective planes in the topological category.

\subsection*{Acknowledgements} The second author would like to thank Adrien Rodau for a question which motivated part of this work. We would also like to thank Erwan Brugall\'e, Duncan McCoy, and JungHwan Park for many interesting discussions.

\section{Line arrangements}\label{s:combinatorics}

We recall two classical lemmas about the combinatorics of arrangements.

Let us set up some notation.
Following Hirzebruch~\cite{Hirzebruch:1983}: for $k \ge 0$, let $f_k = \sum_m m^k t_m(\L)$ so that $f_0 = \sum_m t_m$ and $f_1 = \sum_m mt_m$.

Also, recall that if we have a line arrangement which is realised in $\CP$, and we blow $\CP$ up at all singular points of the line arrangement, in the blow-up we see the \emph{proper transform} (also called \emph{strict transform}) of the arrangement, which is a collection of pairwise disjoint spheres. If we blow up at a point $p$, the corresponding exceptional divisor $E_p$, which is a smoothly embedded sphere of self-intersection $-1$, intersects transversely once the proper transform of each line of the arrangement that passes through $p$.

\begin{lemma}\label{l:f1f2}
Let $\A \subset \CP$ be a locally-flat realisation of line arrangement.
Blow up $\CP$ at all the singular points of $\A$, and call $\L$ the proper transform of $\A$.
\[
\L \cdot \L = d - f_1
\]
\end{lemma}

\begin{proof}
The homology class of $\L$ is
\[
dh - \sum m_i e_i,
\]
so
\[
\L \cdot \L = d^2 - \sum_m m^2t_m = d^2-f_2.
\]
However, from the combinatorics we have
\[
\sum_m {m \choose 2} t_m = {d \choose 2},
\]
which in turn implies that 
\[
d^2 - f_2 = d - f_1,
\]
as needed.
\end{proof}

We also recall a simple inequality between the degree of a non-trivial line arrangement and the number of its singular points.

\begin{lemma}[Basterfield, Kelly~\cite{Basterfield-Kelly:1968}]\label{l:largef0}
Suppose that $\L$ is a combinatorial line arrangement of degree $d$.
Then either $\L$ is a pencil or $f_0 \ge d$.
\end{lemma}

We found the beautiful, elementary proof in~\cite{Grunbaum:1972}, and we include it here for completeness.

\begin{proof}
We will argue by contradiction.

For a point $p$ in $\L$, let $m_p$ denote its multiplicity.
For a line $L$ in $\L$, denote with $r_L$ the number of singular points on it.
The key (elementary) observations are two: first, if $p \not\in L$, then $r_L \ge m_p$, and second, for each $p$ there are exactly $d-m_p$ lines of $\L$ that do not contain $p$.

Armed with these observations, we run a double-counting argument.
The sums below are taken either over $L\in\L$, or over the singular points $p$ of $\L$, or over pairs $(p,L)$ of a point $p$ and a line $L$ with $p \not\in L$.
\[
d = \sum_L 1 = \sum_L \frac{f_0-r_L}{f_0-r_L} = \sum_{L \not \ni p} \frac1{f_0-r_L} \ge \sum_{p \not\in L} \frac1{f_0-m_p} = \sum_p \frac{d-m_p}{f_0-m_p}.
\]
If we now suppse that $f_0 > d$, each summand $\frac{d-m_p}{f_0-m_p}$ in the last sum is larger than $\frac{d}{f_0}$, so we have
\[
d \ge \sum_p \frac{d-m_p}{f_0-m_p} > \sum_p \frac{d}{f_0} = d,
\]
a contradiction.
\end{proof}

\section{Divisible line arrangements}\label{s:even}

First we set up some terminology and some notation.
If $q$ is a positive integer, we say that a line arrangement $\L$ of degree $d$ is \emph{$q$-divisible} if $q$ divides $d$ and all multiplicities of singularities of $\L$.
This generalises the notion of even arrangement, as we defined in the introduction: an even arrangement is just a $2$-divisibile arrangement.

In this section we prove the following theorem and its corollaries.

\begin{thm}\label{t:divisible}
Let $\L$ be a $q$-divisible line arrangement for some $q$. Then:
\[
\sum_m (m-{\textstyle \frac{6q}{q+1}})t_m \le d.
\]
In particular, if $q=2$ we have:
\[
\sum_{m \ge 6} (m-4)t_m \le d + 2t_2.
\]
\end{thm}

We immediately obtain the proof of Theorem \ref{thm:evenbound} and one more corollary.

\begin{proof}[Proof of Theorem \ref{thm:evenbound}]
This is now a consequence of Lemma~\ref{l:largef0}: if $\L$ is even and has $t_2 = t_4 = 0$, then either it is a pencil or $f_0 > 2d$.
If $\L$ is not a pencil,
\[
\sum (m-4)t_m \ge 2 \sum t_m = 2f_0 > d,
\]
which contradicts Theorem~\ref{t:divisible}, so $\L$ is a pencil.
\end{proof}

\begin{cor}\label{c:smalldivisibility}
No non-trivial $q$-divisible combinatorial line arrangement can be topologically realised in $\CP$, if $q \ge 7$ is a prime power.
\end{cor}

\begin{proof}
Suppose that $q \ge 7$ and that a $q$-divisible line arrangement $\L$ is topologically realised.
Since $\L$ is $q$-divisible, $t_m = 0$ for $m < q$.
Moreover, $\frac{6q}{q+1} < 6$, so $m-\frac{6q}{q+1} > 1$ for every $m \ge q$.

Combining these observations with Theorem~\ref{t:divisible},
\[
f_0 < \sum_m (m-{\textstyle \frac{6q}{q+1}})t_m \le d,
\]
but this contradicts Lemma~\ref{l:largef0}.
\end{proof}

Before proving Theorem~\ref{t:divisible}, we state and prove a more general lemma.
We still use the terminology from algebraic geometry when we blow up and take the proper transform of a realisation:
in particular, blowing up at all singular points of the configuration makes the proper transform embedded.

\begin{lemma}\label{l:divisiblecover}
Let $q=p^\ell$ be a prime power. Let $\L$ be a topological realisation of a $q$-divisible line arrangement in $\CP$, and $Y$ be the blow-up of $\CP$ at all singularities of $\L$.
Let $S$ be the $2$-sphere obtained by tubing the $2$-spheres in the proper transform of $\L$.
Then there exist a double cover $X \to Y$, ramified over $S$.
The $4$-manifold $X$ satisfies:
\begin{align*}
\sigma(X) &= \frac{(q^2-1)f_1 - 3q^2f_0 - q^2d + 3q^2 + d}{3q},\\
b_2(X) &= q + qf_0.
\end{align*}

If we specialise to $q = 2$:
\begin{align*}
\sigma(X) &= \frac{4-4f_0+f_1-d}2,\\
b_2(X) &= 2 + 2f_0.
\end{align*}
\end{lemma}

\begin{proof}
Since $\L$ is a $q$-divisible, the homology class of $S$,
\[
[S] = dh - \sum_{i=1}^{f_0} m_i e_i,
\]
is divisible by $q$.

We claim that $H_1(Y\setminus S;\Z/q\Z) = \Z/q\Z$, which in turn implies that there is a (unique) double cover of $Y$ ramified over $S$.
(From now on, in this proof we will work with coefficients in $\Z/q\Z$, unless otherwise stated, but omit it from the notation.)
To see this, let $N$ be a small, open tubular neighbourhood of $S$.
By homotopy invariance, Poincar\'e--Lefschetz duality, and excision:
\[
H_1(Y\setminus S) \cong H_1(Y\setminus N) \cong H^3(Y \setminus N, \del N) \cong H^3(Y,N) \cong H^3(Y,S),
\]
and the latter group fits into the long exact sequence of the pair $(Y,S)$:
\[
H^2(Y) \to H^2(S) \to H^3(Y,S) \to H^3(Y) = 0.
\]
Since $H_*(Y;\Z)$ has no $p$-torsion, the leftmost map is the transpose of the map $H_2(S) \to H_2(Y)$.
Since the class $[S]$ is divisible by $q$ and $H_*(Y;\Z)$ has no $p$-torsion, $H_2(Y) = H_2(Y;\Z)\otimes \Z/q\Z$, and therefore $[S] = 0 \in H_2(Y)$.
In turn, this implies that the map $H_2(S) \to H_2(Y)$ vanishes and so that
\[
\Z/q\Z \cong H^2(S) \cong H^3(Y,S) \cong H_1(Y \setminus S),
\]
as we had set out to prove.

Now, $Y$ is a blow-up of $\CP$ at $f_0$ points.
The Euler characteristic of $X$ is:
\[
\chi(X) = q\chi(Y) - (q-1)\chi(S) = q(3+f_0) - 2(q-1) = q + 2 + qf_0.
\]

Since $H_2(Y\setminus S;\Z/q\Z) = \Z/q\Z$, and $q$ is a prime power, by the Goldschmidt lemma~\cite{Hsiang-Szczarba:1970} we have $b_1(X) = 0$, so that
\[
b_2(X) = \chi(X) - 2 = q + qf_0.
\]
Finally, by the G-signature theorem~\cite{Atiyah-Singer:1968} and Lemma~\ref{l:f1f2}:
\begin{align*}
\sigma(X) &= q\sigma(Y) - \frac{q^2-1}{3q} S\cdot S = q(1-f_0) - \frac{q^2-1}{3q}(d-f_1) \\
& = \frac{(q^2-1)f_1 - 3q^2f_0 - q^2d + 3q^2 + d}{3q}.\qedhere
\end{align*}
\end{proof}

We are now ready to prove Theorem~\ref{t:divisible}.

\begin{proof}[Proof of Theorem~\ref{t:divisible}]
Since $\L$ is $q$-divisible, we can take the $q$-fold cyclic cover $X\to Y$ as in Lemma~\ref{l:divisiblecover}.

Within $X$ we have $f_0$ pairwise disjoint surfaces of self-intersection $-q$, namely the preimages of the exceptional divisors of the blow up $Y \to \CP$.
Therefore, $b_2^-(X) \ge f_0$.

On the other hand, by Lemma~\ref{l:divisiblecover},
\begin{align*}
2b_2^-(X) &= b_2(X) - \sigma(X) = q + qf_0 - \frac{(q^2-1)f_1 - 3q^2f_0 - q^2d + 3q^2 + d}{3q},\\
& = \frac{- (q^2-1)f_1 + 6q^2 f_0 + (q^2-1)d}{3q}.
\end{align*}

Let us compare this computation with the inequality $b_2^-(X) \ge f_0$:
\begin{align*}
\frac{- (q^2-1)f_1 + 6q^2 f_0 + (q^2-1)d}{3q} &\ge 2f_0 \Longleftrightarrow \\
\frac{(q^2-1)d}{3q} &\ge \frac{(q^2-1)f_1 - 6q(q-1) f_0}{3q} \Longleftrightarrow\\
\sum_m (m-{\textstyle \frac{6q}{q+1}})t_m &\le d,
\end{align*}
which proves the first part of the statement.

Specialising to $q = 2$ we get the second half.
\end{proof}

\section{The odd case}\label{s:odd}

\begin{thm}\label{t:odd}
Let $\L$ be a line arrangement with all odd multiplicities: $t_{2k}(\L) = 0$ for every $k$. Suppose that $\L$ is smoothly realized in $\mathbb{CP}^2$. Then either $\L$ is a pencil or
\[
\sum_m (m-9)t_m + 1 + 7d \le 0.
\]
\end{thm}

As a direct consequence we obtain a proof of Theorem \ref{thm:oddbound}.

\begin{proof}[Proof of Theorem \ref{thm:oddbound}]
If $t_3 = t_5 = t_7 = 0$, the left-hand side of the inequality in Theorem~\ref{t:odd} is strictly positive.
\end{proof}

\begin{proof}[Proof of Theorem~\ref{t:odd}]
Suppose that $\L$ is not a pencil. Blow up $\CP$ at all singularities of $\L$, to get a 4-manifold $Y$ containing the proper transform of $\L$, which consists of $d$ pairwise disjoint spheres $S_1 \sqcup \dots \sqcup S_d \subset Y$, each of negative self-intersection.

Since $\L$ is an odd line arrangement, the homology class $[S_1] + \dots + [S_d]$ is characteristic in the lattice $\Lambda_Y$ (which is $H_2(Y)$ equipped with its intersection form).
For $k = 1, \dots, d$, let $N_k$ be a small tubular neighbourhood of $S_k$.
The boundary of $N_k$ (with the orientation induced by $N_k$), which we call $L_k$, is diffeomorphic to the lens space $L(s_k, 1)$.

We will be interested in spin structures on $L_k$, and we recall a general fact about them.
If $s_k$ is odd, there is one spin structure on $L_k$, and this extends to the linear plumbing $P_{s_k}$ of $s_k-1$ spheres of self-intersection $+2$ (which in turn admits a unique spin structure, since it is simply-connected and has even intersection form).
If $s_k$ is even, $L_k$ admits two spin structures: one that extends to $N_k$, and the other one that extends to $P_{s_k}$.
For the purposes of this proof, we call the latter spin structure \emph{good}, and denote it with\footnote{We drop the dependence on the lens space, but this will hopefully not create any confusion.} $\fs_{\rm g}$ (the other one will be $\fs_{\rm b}$).
(When $s_k$ is odd, the unique spin structure on $L_k$ will also be \emph{good}, and will be denoted by $\fs_{\rm g}$.)

The good spin structure on $L_k$ is characterised by its Heegaard Floer correction term (or $d$-invariant).
Indeed, since $L_k$ is a rational homology spheres, spin structures on $L_k$ correspond to self-conjugate spin$^c$ structures on $L_k$, and to a spin$^c$ structure $\ft$ on a rational homology sphere $Y$ one can associated the rational number $d(Y,\ft)\in \Q$.
It turns out that
\[
d(L(n,1),\fs_{\rm g}) = -\frac{n-1}4,\qquad d(L(2n,1),\fs_{\rm b}) = \frac14.
\]

We claim that there exist a spin structure on $Y$ that restricts to the good spin structure on $L_k$ for each $k$.
To see this, start by tubing all spheres $S_1, \dots, S_d$, to obtain a characteristic sphere $S \subset Y$.
Since $\L$ is an odd line arrangement, $Y\setminus S$ is spin.
As in the proof of Lemma~\ref{l:divisiblecover}, since $\L$ is odd, we have $H_1(Y\setminus S; \Z/2\Z) = 0$, so that the spin structure on $Y \setminus S$ is unique.
Since $Y$ itself is \emph{not} spin, the spin structure on $Y$ does not extend to the neighbourhood $N$ of $S$, so it is necessarily good.

We now `undo the tubing' of $S$ to get back to $S_1 \sqcup \dots \sqcup S_d$.
To be more precise, we can assume (up to isotopy) that $S$ is contained in the union $U$ of the neighbourhoods $N_1, \dots, N_d$ and of the tubular neighbourhoods of the arcs we used to guide the tubing.
We can also assume that the tubular neighbourhood $N$ of $S$ mentioned above is contained in $U$.
Call $L$ the boundary of $N$.

Then $W := U\setminus N$ is a cobordism from $L$ to $L_1 \sqcup \dots \sqcup L_d$, with the following properties:
\begin{itemize}
\item it inherits a spin structure from $Y\setminus S$, which in turn restricts to $\fs_{\rm g}$ on $L$;
\item it is negative definite, since it is contained in $N_1 \# \dots \# N_d$, which is a connected sum of neighbourhoods of spheres of negative self-intersection, and has $b_2(W) = d-1$.
\end{itemize}
Denote with $\ft_k$ the restriction of the spin structure of $W$ to $L_k$, which we view as a self-conjugate spin$^c$ structure on $L_k$.
Applying the Ozsv\'ath--Szab\'o inequality\footnote{Technically, after digging a tunnel between $L_k$ and $L_{k+1}$ for each $k < d$.} to the spin$^c$ structure on $W$ with trivial first Chern class (that is to say, to any spin structure), we have
\begin{align*}
-\frac{(\sum_k s_k)-1}4 = d_L(L,\fs_{\rm g}) &\le {\textstyle \frac14}b_2(W) + d(L_1,\ft_1) + \dots + d(L_d,\ft_d) \\
&\le \frac{d-1}4 - \frac{\sum (s_k-1)}4 = -\frac{(\sum_k s_k)-1}4,
\end{align*}
which forces all inequalities $d(L_k,\ft_k) \le -\frac{s_k-1}4$ to be equalities, which in turn implies that $\ft_k = \fs_{\rm g}$ for every $k$.

Now we glue $P_{s_1}\sqcup \dots \sqcup P_{s_d}$ to $Y \setminus (N_1 \sqcup \dots \sqcup N_d)$ along their common boundary.
Call $X$ the resulting closed $4$-manifold.

Since $Y' := Y \setminus (N_1 \sqcup \dots \sqcup N_d)$ has a spin structure that restricts to $\fs_{\rm g}$ on each of its boundary component, $X$ is spin.

We now compute $b_2(X)$ and $\sigma(X)$.
First, by additivity of the Euler characteristic,
\begin{align*}
\chi(X) &= \chi(Y') + \sum_k \chi(P_{s_k}) = \chi(Y) - \sum_k \chi(N_k) + \sum_k \chi(P_{s_k})\\
& = 3 + f_0 - 2d + \sum_k |s_k| = 3 + f_0 - 2d + f_1 - d = 3-3d+f_0+f_1.
\end{align*}
Moreover, $b_1(X) = 0$ since $H_1(Y;\Q) = 0$ (this is because the spheres $S_1, \dots, S_d$ are linearly independent in $H_2(Y)$: they each have negative self-intersection and they are pairwise disjoint) and $H_1(P_{s_k}) = 0$ for each $k$.
Then $b_2(X) = 1-3d+f_0-f_1$.

Now, by Novikov additivity,
\begin{align*}
\sigma(X) &= \sigma(Y') + \sum_k \sigma(P_{s_k}) = \sigma(Y) - \sum_k \sigma(N_k) + \sum_k \sigma(P_{s_k})\\
& = 1-f_0 + d + \sum_k (|s_k|-1) = 1 - d - f_0 + f_1.
\end{align*}

Applying Furuta's 10/8-Theorem, $5\sigma(X) \le 4b_2(X)$, we obtain:
\[
5\cdot (1-d - f_0 + f_1) \le 4\cdot(1-3d+f_0+f_1) \Rightarrow f_1 - 9f_0 \le -1 - 7d,
\]
which is just
\[
\sum_m (m-9)t_m \le -1-7d.\qedhere
\]
\end{proof}

\section{Configurations and divisible arrangements}\label{s:configurations}

Recall from the introduction that a combinatorial $(n_k)$-configuration must satisfy $n \ge k^2-k+1$.
In this section, we push the bound a bit higher for configurations that are topologically realisable.

\begin{thm}\label{t:configurations}
If a line arrangement of degree $n$ and with $t_k \ge n$ is realisable, then:
\[
n \ge \left\{\begin{array}{ll}
k^2 - 4 & \mbox{if } n \equiv 0,\, k \equiv 0 \pmod 2\\
k^2 - 5 & \mbox{if } n \equiv 1,\, k \equiv 0 \pmod 2\\
k^2 - 5 & \mbox{if } n \equiv 0,\, k \equiv 1 \pmod 2\\
k^2 - 4 & \mbox{if } n \equiv 1,\, k \equiv 1 \pmod 2\\
\end{array}
\right.
\]
Therefore, regardless of the parity of $n$ and $k$,
\[
n \ge k^2-5.
\]
In particular, this applies to $(n_k)$-configurations.
\end{thm}

We give an auxiliary definition.

\begin{defn}
Let $d$, $t$, $t'$, and $\ell$ be positive integers, with $d$ and $\ell$ even.
A \emph{$(d,t,\ell)$-arrangement} is a combinatorial line arrangement of degree $d$, and which has $t$ points of multiplicity $\ell$ and whose only other singularities are double points.

A \emph{$(d,t,t',\ell)$-arrangement} is a combinatorial line arrangement of degree $d$, and which has $t$ points of multiplicity $\ell$, $t'$ points of multiplicity $\ell-2$, and whose only other singularities are double points.
\end{defn}

We stipulate that a $(d,t,\ell)$-arrangement and a $(d,t,t',\ell)$-arrangement always have $d$ and $\ell$ even, and in particular they are always even arrangements.
Note also that a $(d,t,\ell)$-arrangement has exactly ${d \choose 2} - t{\ell \choose 2}$ double points, while a $(d,t,t',\ell)$ arrangement has ${d \choose 2} - t{\ell \choose 2} - t'{\ell-2 \choose 2}$ double points.

\begin{proof}[Proof of Theorem~\ref{t:configurations}]
Let $(\cP, \L)$ be a realisable $(n_k)$-configuration.
As noted above, $n \ge k^2-k+1$, so the inequalities are already true combinatorially for $k < 6$.
So we can suppose $k \ge 6$.

By perturbing the points of multiplicities different than $k$ and $2$ if necessary, we obtain a topological realisation of a line arrangement $\L'$ with $t_k(\L') = n$ and $t_2(\L') = {n \choose 2} - n {k \choose 2}$.

By applying Theorem~\ref{t:divisible} to a $(d,t,\ell)$-arrangement, we obtain
\[
(\ell-4)t \le d + (d^2-d) - d(\ell^2-\ell),
\]
which is equivalent to
\begin{equation}\label{e:dtl}
d^2 \ge (\ell^2-4)t.
\end{equation}
We now construct a $(d,t,\ell)$-arrangement starting from $\L'$, in three of the four cases.

\begin{itemize}
\item If $n$ and $k$ are both even, the arrangement $\L'$ defined above is an $(n,n,k)$-arrangement, and~\eqref{e:dtl} simplifies to $n\ge k^2-4$, as required.

\item Suppose $n$ is odd and $k$ is even.
As on average a line in the arrangement contains $k$ points of multiplicity $k$, we can find a line in $\L'$ that contains at least $k$ points\footnote{Any line would work if we worked with a configuration.}.
Remove this line from the arrangement. 
All intersection points that were on this line either disappeared (because they were double points) or have now multiplicity $k-1$, which is odd. 
Perturb each of the latter points to obtain an $(n-1,t,t',k)$-arrangement with $t + t' = n$ and $t \le n-k$.
Applying Theorem~\ref{t:divisible} and $t \le n-k$ we obtain:
\[
(k-4)t + (k-6)t' \le n-1 + 2{d \choose 2} - 2t{k \choose 2} - 2t'{k-2 \choose 2}.
\]
Since $t \ge n-k$,
\[
(k-4)(n-k) + (k-6)k \le (k-4)t + (k-6)t'
\]
and
\[
2{d \choose 2} - 2t{k \choose 2} - 2t'{k-2 \choose 2} \le 2{n \choose 2} - 2(n-k){k \choose 2} - 2k{k-2 \choose 2},
\]
so that
\[
(k-4)(n-k) + (k-6)k \le n-1 + 2{n \choose 2} - 2(n-k){k \choose 2} - 2k{k-2 \choose 2}.
\]
From this one obtains:
\[
n^2 - (k^2-2)n + 4k^2-4k+1 \le 0,
\]
and solving the degree-2 inequality in $n$ yields
\[
n \ge \frac{k^2-2 + \sqrt{k^4-20k^2 + 16k}}2 > \frac{k^2-2 + k^2-12}2 = k^2-7,
\]
where we used $k \ge 6$ to say that $\sqrt{k^4-20k^2 + 16k} > k^2-12$.
Since $n$ is odd and $k$ is even, $n > k^2-7$ implies $n \ge k^2-5$.
(Note that the other half-line of solutions of the quadratic inequality is not relevant, since the other solution of the quadratic equation is smaller than $1$.)

\item If $n$ is even and $k$ is odd, we remove $k-1$ concurrent lines from $\L'$ and we perturb all remaining points of multiplicity $k$ into a point of multiplicity $k-1$ and $k-1$ points of multiplicity 2. The resulting arrangement is an $(n-k+1,n-1,k-1)$-arrangement, since we lowered the degree by $k-1$, we removed a single point of multiplicity $k$, and turned all others into points of multiplicity $k-1$ (either by removing a line through them or by perturbing). Equation~\eqref{e:dtl} now gives:
\[
(n-k+1)^2 \ge (k^2-2k-3)(n-1)  \Longleftrightarrow  n^2 - (k^2-5)n + 2k^2-4k-2.
\]
Solving for $n$, we get:
\[
n \ge \frac{k^2 - 5 + \sqrt{k^4-18k^2 + 16k + 33}}2 > \frac{k^2 - 5 + k^2- 9}2 = k^2-7.
\]
Here we used $k \ge 5$ to say that $16k+33 > 81$, which in turn implies $\sqrt{k^4-18k^2 + 16k + 33} > k^2-9$.
(Similarly to the previous case, the other solution of the quadratic equation is smaller than $4$, so not geometrically relevant.)
Since $n$ and $k$ have opposite parities, we actually obtain $n \ge k^2-5$.

\item If $n$ and $k$ are both odd, we remove all lines passing through a point of multiplicity $k$ and perturb all remaining points of multiplicity $k$ as above, to get an $(n-k,n-1,k-1)$-arrangement.
Equation~\eqref{e:dtl} yields:
\[
(n-k)^2 \ge (n-1)(k^2-2k-3) \Longleftrightarrow n^2 - (k^2-3)n + 2k^2-2k-3.
\]
We solve for $n$ and obtain:
\[
n \ge \frac{k^2 - 3 + \sqrt{k^4-14k^2 + 8k + 21}}2 > \frac{k^2 - 5 + k^2- 7}2 = k^2-6,
\]
where we used $k \ge 5$ to say that $8k+21 > 49$, and so $\sqrt{k^4-14k^2 + 8k + 21} > k^2-7$.
(As above, the other solution of the quadratic equation is smaller than 4, so not geometrically relevant.)
Since $n$ and $k$ have the same parity, we have $n \ge k^2-4$.
\end{itemize}
This concludes the proof.
\end{proof}

A special case of configurations is given by finite projective planes.

\begin{defn}
A \emph{finite projective plane} is a pair $(\mathcal{P}, \L)$ where $\mathcal{P}$ is a finite set (whose elements are called \emph{points}) and $\mathcal{L}$ is a collection of subsets of $\mathcal{P}$ (called \emph{lines}), such that:
\begin{itemize}
\item any two lines intersect in one point;
\item for every two points there exists a unique line containing them;
\item there exist four points that are not in the same line.
\end{itemize}
Two finite projective planes $(\mathcal{P}, \L)$ and $(\mathcal{P}', \L')$ are \emph{isomorphic} if there exists two bijections $\pi \co \mathcal{P}_1 \to \mathcal{P}_2$ and $\lambda\co \L_1 \to \L_2$ such that $\lambda(L) \cap \lambda(L') = \pi(L \cap L')$ for every $L, L' \in \L$.
\end{defn}

It is known that if $(\mathcal{P}, \L)$ is a finite projective plane, then $\mathcal{P}$ and $\L$ both have cardinality $q^2+q+1$ for some integer $q$, that each line contains $q+1$ points and that, dually, every point belongs to $q+1$ lines.
The integer $q$ is called the \emph{order} of the finite projective plane.
In fact, a pair $(\mathcal{P}, \L)$ is a finite projective plane of order $q$ if and only if it is an $((q^2+q+1)_{q+1})$-configuration.

\begin{ex}
Let $q$ be a prime power and identify $\F_q\P^2$, the projective plane over the finite field of order $q$. Then the lines of $\F_q\P^2$, form a finite projective plane of order $q$.
It is known that there exist finite projective planes of order $9$ that are not isomorphic to $\F_9\P^2$~\cite{Hall-Swift-Killgrove:1959}.
It is not known whether there exist finite projective planes of order $q$ for some $q$ which is not a prime power.
\end{ex}

\begin{thm}\label{t:nofinite}
No finite projective plane can be topologically realised in $\CP$.
\end{thm}

We give two proofs.

\begin{proof}[First proof]
Let $(\cP,\L)$ be a finite projective plane, and let $q$ be its order.
Then $(\cP,\L)$ is an $((k^2-k+1)_k)$-configuration, where $k = q+1$.
If $q = k-1 > 6$, such a configuration is not realisable by Theorem~\ref{t:configurations}, so we only need to discuss the cases $q = 2,\dots, 6$.

There are no finite projective planes of order 6~\cite{Lam:1991}\footnote{Even if one existed, it would not be realisable by~\cite[Theorem~3]{Golla:2024}, as it would have degree 43 and 43 points of multiplicity 7, but $43\cdot(7-1) \neq 43-1 \pmod{16}$. It would also contradict Theorem~\ref{t:odd} above.}.
For $q = 2,\dots, 5$, there is a unique finite projective plane, namely $\mathbb{F}_q\P^2$ ($q$ is always a prime power in this range).
These are not realised by earlier work of~\cite{Ruberman-Starkston:2016}.
\end{proof}

\begin{proof}[Second proof]
Let $(\cP,\L)$ be a finite projective plane, and let $q$ be its order.
As in the proof above, if $q \le 6$ there is no topological realisation by~\cite{Ruberman-Starkston:2016} and~\cite{Lam:1991}.

Suppose then that $q \ge 6$. Choose a point $x \in \cP$ and remove all the lines through $x$.
The resulting arrangement $\L_x$ is $q$-divisible, since it contains $q^2$ lines and all its points have multiplicity $q$.
So, by Corollary~\ref{c:smalldivisibility}, $\L_x$ is not realisable, and therefore neither is $\L$.
\end{proof}

In fact, something similar is true for line arrangements whose intersection points have the same multiplicity.

\begin{thm}
Let $m > 7$.
There is no topologically realisable line arrangement in $\CP$ with all singular points of multiplicity $m$, except for the pencil of degree $m$.
\end{thm}

\begin{proof}
If $m$ is odd, this follows directly from Theorem~\ref{t:odd}.

Suppose now that $m$ is even and call $d$ the degree of such a line arrangement.
If $d$ is even, the arrangement is even, and by Theorem~\ref{thm:evenbound} it cannot be topologically realised.

Suppose now that $d$ is odd.
On average, lines of the arrangement contain $\frac{d-1}{m-1}$ singular points.
Choose a line that contains at least $t'$ singular points, with $t' \ge \frac{d-1}{m-1}$, and remove it from the arrangement.
Perturb each of the $t'$ points of intersection that used to lie on this line to get $t'$ points of multiplicity $m-2$ and $t'(m-1)$ points of multiplicity $2$.
This yields a $(d-1,t,t',m)$-arrangement, where $t+t' = \frac{d^2-d}{m^2-m}$ and $t' \ge \frac{d-1}{m-1}$.
Applying Theorem~\ref{t:divisible}, and using $t' \ge \frac{d-1}{m-1}$ as in the proof of Theorem~\ref{t:configurations}, we get
\[
(m^2-4)\left(\frac{d^2-d}{m^2-m}-\frac{d-1}{m-1}\right) + (m^2-4m)\frac{d-1}{m-1} \le (d-1)^2.
\]
Divide by $d-1$ and multiply by $m^2-m$, we get
\[
d \le  \frac{3m(m-1)}{m - 4} = 3m+3 + \frac{36}{m-4}.
\]
On the other hand, if the arrangement is not a pencil, then it contains at least $(m-1)^2 + 2 = m^2-2m+3$ intersection points.
This is because we can choose two points.
For each of them we can find at least $(m-1)$ lines not containing the other, which accounts for $(m-1)^2$ intersection points.

Now,
\[
m^2-2m+3 \le d \le 3m+3 + \frac{36}{m-4} \Rightarrow m(m-5)(m-4) \le 36,
\]
which implies $m < 7$, as claimed.
\end{proof}

\bibliographystyle{amsalpha}
\let\MRhref\undefined
\bibliography{bib}

\end{document}